\documentclass[a4,11pt]{amsart}

\usepackage{mystyle}
\usepackage{hyperref}
\usepackage[capitalize,nameinlink]{cleveref}

\newcommand*{\inv}{^{-1}}

\newcommand*{\cQ}{{\mathcal Q}}
\newcommand*{\cD}{{\mathcal D}}
\newcommand*{\cF}{{\mathcal F}}
\newcommand*{\cL}{{\mathcal L}}
\newcommand*{\cE}{{\mathcal E}}

\newcommand{\av}[1]{\left\Vert #1\right\Vert}

\newtheorem{theorem}{Theorem}[section]
\newtheorem{proposition}[theorem]{Proposition}
\newtheorem{lemma}[theorem]{Lemma}
\newtheorem{corollary}[theorem]{Corollary}
\newtheorem*{remark}{Remark}

\theoremstyle{definition}
\newtheorem{definition}[theorem]{Definition}

\begin{document}
\title[A non-linear characterization of stochastic completeness]{A non-linear characterization of stochastic completeness of graphs}

\author{Marcel Schmidt}
\address{Mathematisches Institut, Universit\"at Leipzig, 04109  Leipzig, Germany}
\email{marcel.schmidt@math.uni-leipzig.de}

\author{Ian Zimmermann}
\address{Mathematisches Institut, Friedrich-Schiller-Universit\"at Jena, 07743 Jena, Germany}
\email{ian.zimmermann@uni-jena.de}

\begin{abstract}
We study non-linear Schrödinger operators on graphs. We construct minimal nonnegative solutions to corresponding semi-linear elliptic equations and use them to introduce the notion of stochastic completeness at infinity in a non-linear setting. We provide characterizations for this property in terms of a semi-linear Liouville theorem. It is employed to establish a non-linear characterization for stochastic completeness, which is a graph version of a recent result on Riemannian manifolds.   
\end{abstract}

\maketitle

\section{Introduction} \label{sec_introduction}

A Riemannian manifold $M$ is called {\em stochastically complete} if the minimal Brownian motion on it has infinite lifetime. In terms of the semigroup induced by the Friedrichs extension of the Laplacian (which we simply denote by $\Delta$), this is equivalent to the {\em conservation property} 
$$e^{t\Delta } 1  = 1, \text{ all } t > 0. $$
By standard arguments it can also be expressed through the resolvent and then reads $\alpha (\alpha - \Delta)^{-1} 1 = 1$, all $\alpha > 0$. 

There is a huge amount of literature on geometric criteria involving volume growth and curvature bounds that ensure stochastic completeness on manifolds \cite{gaf59,kha60,aze74,KarpLi,Yau78,Ich82,gri86,Hsu89,Tak89,Dav92,Mur94,stu94,Ish01,Mur03,PRS03,RSV05,MV13} (this list is not comprehensive and we refer to  \cite{gri99} for further background). Beyond manifolds, in recent years there has been an increased interest in such geometric criteria for stochastic completeness of nonlocal operators e.g. for discrete graph Laplacians \cite{Woj08,Woj11,Hua11,Huang12,GHM12,Fol14,Huang14, HKS20} (we refer to the notes after Chapter~14 in \cite{KLW21} for a more detailed account of the subject).


%

The mentioned geometric criteria for stochastic completeness are based on analytic and probabilistic characterizations and we refer to the survey \cite{gri99} for a comprehensive overview (see also \cite{KLW21} for the discrete case). Here, we only mention two characterizations: Due to the parabolic maximum principle the semigroup $(e^{t\Delta})_{t>0}$ induces minimal nonnegative solutions to the heat equation 
$$\begin{cases}
   \partial_t u = \Delta u &\text{on } (0,\infty) \times M \\
   u(0,\cdot) = u_0&\text{on } M
  \end{cases}
$$
 for given $u_0 \geq 0$. Similarly, by the elliptic maximum principle the resolvent  $(\alpha - \Delta)^{-1}$, $\alpha > 0$, yields minimal nonnegative solutions to the Poisson equation $(\alpha -\Delta ) f = h$ for given $h \geq 0$ and $\alpha > 0$. With these observations stochastic completeness (more precisely the conservation property) can be seen to be equivalent to uniqueness of bounded solutions to the heat equation/poisson equation.    

 In recent years these characterizations of stochastic completeness have been expanded to more general equations in two different directions.
 
 {\em 1. Schrödinger operators:} If one considers the Schrödinger operator $\Delta - V$ with nonnegative potential $V\neq 0$, then the conservation property of the semigroup does not hold. In this case, one always finds $e^{t(\Delta-V)}1 < 1$ for all $t > 0$, but the heat equation and the Poisson equation with the operator $\Delta-V$ (instead of $\Delta$) might still have unique bounded solutions. It was first observed in \cite{kel12} for graphs and then extended to Riemannian manifolds in \cite{mas20} that uniqueness of bounded solutions to the heat equation/Poisson equation with respect to $\Delta-V$ is equivalent to a {\em generalized conservation property}, which is also referred to as {\em stochastic completeness at infinity}.  
 
 {\em 2. Non-linear equations:} A different perspective is taken in \cite{gri20, gri23}. There, it is proven that stochastic completeness of $M$ is equivalent to: 
 
 \begin{itemize}
  \item For one strictly increasing and sufficiently regular $\psi\colon [0,\infty) \to [0,\infty)$ and all $u_0 \in L_+^\infty(M)$, the non-linear filtration equation
 $$\begin{cases}
   \partial_t u = \Delta \psi(u) &\text{on } (0,\infty) \times M \\
   u(0,\cdot) = u_0&\text{on } M
  \end{cases}
$$
has unique bounded solutions. 

\item For one strictly increasing continuous $\varphi \colon [0,\infty) \to [0,\infty)$ with $\varphi(0) = 0$ the semi-linear elliptic equation $\Delta f = \varphi(f)$ only has trivial nonnegative bounded (sub)solutions. 
 \end{itemize}

 The aim of this paper is to extend the elliptic part of both perspectives (generalized conservation property, non-linear characterizations of stochastic completeness) to non-linear Schrödinger operators on graphs. To discuss our main results we start with a reformulation of the semi-linear elliptic equation. If $\varphi \colon \IR \to \IR$ is strictly increasing with $\varphi(0) = 0$,  the equation $\Delta f = \varphi(f)$ can be reformulated as $\varphi^{-1}(\Delta f) - f = 0$. We replace $f$ by $-g$,  $-\Delta$ by the weighted graph Laplacian $\cL$, additionally introduce a nonnegative potential $W$ with $\inf W > 0$, and study the semi-linear equation 
 $$\varphi^{-1} (\cL g) + W g = h.$$
 For nonnegative $h$ we use a non-linear comparison principle and variational methods to construct minimal (super)solutions, see Theorem~\ref{theorem:minimality resolvent}, which is our first main result.  We employ these minimal (super)solutions to characterize the triviality of nonpositive bounded supersolutions to the equation $\varphi^{-1} (\cL g) + W g = 0$ via a non-linear conservation property, see Theorem~\ref{theorem:uniqueness of bounded subsolutions}, which is our second main result. This extends the main result of \cite{kel12} to certain non-linear operators and gives rise to a definition of stochastic completeness at infinity in a semi-linear setting, see Definition~\ref{definition:nonlinear stochastic completeness}. We provide a geometric condition that ensures stochastic completeness at infinity for big enough $W$, see Proposition~\ref{prop:geometric criterion}
 
 We also use minimal supersolutions to establish (and slightly generalize) the main elliptic result of \cite{gri23} for graphs. We show that stochastic completeness is equivalent to the triviality of bounded nonnegative solutions to  $ -\cL f  = \varphi(W f)$ for one strictly increasing continuous $\varphi \colon \IR \to \IR$ with $\varphi(0) = 0$ and one bounded $W$ with $\inf W >0$ (or all $W$ with $\inf W >0$, \cite{gri23} treats the case $W = 1$), see Theorem~\ref{theorem:stochastic completeness}, which is our third main result. It actually goes beyond \cite{gri23} by also providing non-linear conservation properties that characterize stochastic completeness.  
 
We wrote this paper specifically for the graph case for two reasons. On the one hand, as discussed above, stochastic completeness of graphs has gained quite some attention in recent years. On the other hand, the graph Laplacian is a non-local operator and certain tools (chain rules and local estimates) are not available. We wanted to provide a proof of one of the main results of  \cite{gri20,gri23}  that is less reliant on local PDE techniques but rather uses variational methods (plus a comparison principle). This approach makes it possible to extend the results to other operators, such as e.g. fractional Laplacians or Laplacians on metric graphs. 

A characterization of stochastic completeness of graphs in terms of uniqueness of bounded solutions to  non-linear filtration equations is beyond the scope of this paper.  The reason is the lack of a sufficiently general theory of positive solutions to non-linear parabolic equations on graphs. A first step towards such a theory has been recently taken in \cite{BSW22}, where $\ell^1$-solutions for the porous medium equation on graphs are considered. We note however, that a characterization of stochastic completeness via uniqueness of bounded nonnegative subsolutions to filtration equation with respect to the special nonlinearity $\psi(t) = \log(t)$ (which is somewhat skew to the assumptions in \cite{gri23}, which require $\psi(0) = 0$) is already treated in \cite{Ma22}.

{\bf Acknowledgements.} M.S. acknowledges financial support of the DFG within the priority programme ’Geometry at Infinity’.

\section{Preliminaries}

In this section we introduce the setting of operators on weighted discrete graphs. We closely follow \cite{KLW21} and refer to it for claims which are not explained in detail. 

\subsection{Graphs and formal objects}

Throughout we let $X \neq \emptyset$ be a countable set. We write $C(X)$ for the real-valued functions on $X$ and $C_c(X)$ for the real-valued functions of finite support.  Every function $m \colon X \to (0,\infty)$ induces a measure on all subsets of $X$, which we also denote by $m$, via
$$m(A) = \sum_{x \in A} m(x),  \quad A \subseteq X.$$
We call such a pair $(X,m)$ {\em discrete measure space}. For $1 \leq p < \infty$ we let 
$$\ell^p(X,m) = \{f \colon X \to \IR\mid \sum_{x \in X} |f(x)|^pm(x) < \infty\}$$
denote the real-valued discrete Lebesgue space and equip it with the norm $$\av{f}_p = \left(\sum_{x \in X} |f(x)|^pm(x)\right)^{1/p}.$$ 
The space of bounded real-valued functions is denoted by $\ell^\infty(X)$ and equipped with the supremum norm $\av{\cdot}_\infty$. 

A {\em weighted graph} over $X$ is a symmetric function $b\colon X \times X \to [0,\infty)$ with $b(x,x) = 0$, $x \in X$, such that at all $x \in X$ the {\em weighted vertex degree}  is finite, i.e., 
$$\deg(x) := \sum_{y \in X} b(x,y) <  \infty. $$
The elements of $X$ are called {\em vertices}. Two vertices $x,y \in X$ are thought to be connected by an {\em edge} if $b(x,y) > 0$, in which case we write $x \sim y$. A {\em path} is a finite of infinite sequence of vertices $\gamma = (x_1,x_2,\ldots)$ such that $x_j \sim x_{j+1}$, for all $j =1,2,\ldots$. We say that a subset $U \subseteq X$ is {\em connected} if for every $x,y \in U$ there exists a path $\gamma$ in $U$ that contains $x$ and $y$. 

The {\em graph energy} is the functional 
$$\cQ \colon  C(X) \to [0,\infty], \quad \cQ(f) = \frac{1}{2} \sum_{x,y \in X} b(x,y) (f(x) - f(y))^2 $$
and the corresponding space of functions of finite energy is given by 
$$\cD = \{f \in C(X) \mid \cQ(f) < \infty\}.$$
By polarization $\cQ$ induces a bilinear form on $\cD$, which we also denote by $\cQ$. In this sense, we have $\cQ(f) = \cQ(f,f)$ for $f \in \cD$. 

By our assumption on the finiteness of the weighted vertex degree we have $C_c(X) \subseteq \cD$. We let $\cD_0$ be the space of those $f \in \cD$ for which there exists a sequence $(f_n)$ in $C_c(X)$ with $f_n \to f$ pointwise, as $n \to \infty$, and  $\lim_{n \to \infty} \cQ(f-f_n) = 0$.  Moreover, we let 
$$\cQ_0 \colon C(X) \to [0,\infty],\quad \cQ_0(f) = \begin{cases}
                                                     \cQ(f) &\text{if } f \in \cD_0\\
                                                     \infty &\text{else}
                                                    \end{cases}.
$$
\begin{lemma}\label{lemma:pointwise semicontinuity}
The functionals $\cQ$ and $\cQ_0$ are lower semicontinuous with respect to pointwise convergence. In particular, if  $(f_n)$ in $C(X)$ converges pointwise to $f \in C(X)$ and $(\cQ_0(f_n))$ is bounded, then $f \in \cD_0$ and 
$$\cQ_0(f) \leq \liminf_{n \to \infty} \cQ_0(f_n).$$
\end{lemma}
\begin{proof}
 The lower semicontinuity of $\cQ$ is a simple application of Fatou's lemma. For the lower semicontinuity of $\cQ_0$ we refer to \cite[Proposition~3.5]{MS23}.
\end{proof}

\begin{lemma}\label{lemma:compatibility with contractions}
 Let $f \in C(X)$ and let $C \colon \IR \to \IR$ be a normal contraction. Then 
 $$\cQ(C \circ f) \leq \cQ(f) \text{ and } \cQ_0(C \circ f) \leq \cQ_0(f).$$
 In particular, for $f,g \in C(X)$ we have 
 $$\cQ(f \wedge g) \leq \cQ(f) + \cQ(g),$$
 where $f \wedge g = \min\{f,g\}$. 
\end{lemma}
\begin{proof}
 For $\cQ$ this is a direct consequence of the definition. For $\cQ_0$ we have to take into account where it is finite and  make use of the lower semicontinuity. More precisely, for $f \in \cD_0$ (for other $f$ there is nothing to show) we choose a sequence $(f_n)$ in $C_c(X)$ with $f_n \to f$ pointwise and $\cQ(f-f_n) \to 0$ (which also yields $\cQ(f_n) \to \cQ(f) = \cQ_0(f)$). Then $C(f_n) \to C(f)$ pointwise and  $C(f_n) \in C_c(X) \subseteq \cD_0$. Using $\cQ = \cQ_0$ on $\cD_0$, the pointwise lower semicontinuity of $\cQ_0$ and the inequality for $\cQ$, we obtain 
 \begin{align*}
  \cQ_0(C \circ f) &\leq \liminf_{n\to \infty} \cQ_0(C \circ f_n) =\liminf_{n\to \infty} \cQ(C \circ f_n)\\
  &\leq \liminf_{n\to \infty} \cQ(f_n) = \cQ_0(f). 
 \end{align*}

 'In particular'-part:  Applying the already established inequality to the absolute value yields $\cQ(|f|) \leq \cQ(f)$ for all $f \in C(X)$. It is well-known that for quadratic forms this inequality leads to $\cQ(f \wedge g) \leq \cQ(f) + \cQ(g)$, see e.g. the proof of \cite[Lemma 1.22]{KSW20} for the required computations.  
\end{proof}

As an application of both properties we show that nonnegative functions in $\cD_0$ can be approximated by nonnegative functions in $C_c(X)$ satisfying a certain bound. 

\begin{lemma}\label{lemma:bounded approximation}
 Let $0 \leq f \in \cD_0$. Then there exists a sequence $(f_n)$ in $C_c(X)$ with $0 \leq f_n \leq f$, $f_n \to f$ pointwise and $\cQ(f_n - f) \to 0$. 
\end{lemma}
\begin{proof}
We first choose a sequence $(g_n)$ in $C_c(X)$ with $g_n \to f$ pointwise and $\cQ(f - g_n) \to 0$. We then consider the sequence $h_n := (g_n)_+ \wedge f$. Then $0\leq h_n \leq f$ and $h_n \to f$ pointwise. From Lemma~\ref{lemma:compatibility with contractions} we infer 
$$\cQ(h_n) \leq \cQ((g_n)_+) + \cQ(f) \leq \cQ(g_n) + \cQ(f).$$
Since $\cQ(g_n - f) \to 0$ and $\cQ$ is a quadratic form, the sequence $(h_n)$ is $\cQ$-bounded. By the Banach-Saks theorem it has a subsequence $(h_{n_j})$ such that the sequence $(f_N)$ defined by
$$f_N =  \frac{1}{N}\sum_{j=1}^N h_{n_j}$$
is $\cQ$-Cauchy. By construction it satisfies $0\leq f_N \leq f$ and $f_N \to f$ pointwise. Moreover, the pointwise lower semicontinuity  of $\cQ$ yields
$$\cQ(f - f_M) \leq \liminf_{N \to \infty} \cQ(f_N - f_M).$$
The right side of this inequality becomes small for large $M$ because $(f_N)$ is $\cQ$-Cauchy. 
\end{proof}

  The {\em formal Laplacian} $\cL = \cL_{b,m}$ of a graph $b$ over the discrete measure space $(X,m)$ has the domain 
  $$\cF = \cF_b = \{f \in C(X) \mid \sum_{y \in X} b(x,y)|f(y)| < \infty \text{ for all } x \in X\},$$
on which it acts by
$$\cL f(x) =   \frac{1}{m(x)} \sum_{y \in X} b(x,y)(f(x) - f(y)).$$
Our assumption on the finiteness of the weighted vertex degree implies $\ell^\infty(X) \subseteq \cF$. The formal Laplacian is related to the quadratic form $\cQ$ through  the following Green formula, see e.g. \cite[Proposition~1.4]{KLW21}.

\begin{lemma}[Green's formula]\label{lemma:greens formula}
 $\cD \subseteq \cF$ and for $f \in \cD$ and $\varphi \in C_c(X)$ we have 
 $$\cQ(f,\varphi) = \sum_{x \in X} (\cL f)(x) \varphi(x) m(x).$$
\end{lemma}

%
%

\subsection{Dirichlet forms and stochastic completeness (at infinity)} \label{subsection:stochastic completeness}

 In this subsection we briefly recall the concept of stochastic completeness and stochastic completeness at infinity for graphs. Let $c \colon X \to [0,\infty)$. We consider the quadratic form $Q = Q_{b,c}$ on $\ell^2(X,m)$ with domain 
$$D(Q) = \{u  \in \cD_0 \cap \ell^2(X,m) \mid \sum_{x \in X} u(x)^2 c(x) < \infty\},$$
on which it acts by
$$Q(u) = \cQ(u) + \sum_{x\in X} u(x)^2 c(x).$$
It is a regular Dirichlet form and by Green's formula the associated nonnegative self-adjoint operator $L = L_{b,c,m}$ on $\ell^2(X,m)$ satisfies $D(L) \subseteq \cF$ with
$$Lf(x) = \cL f(x) + \frac{c(x)}{m(x)} f(x)$$
for $f \in D(L)$, $x \in X$. 

\begin{remark}
 For more details see \cite[Chapter~1]{KLW21}. The operator $L$ can be thought of being the self-adjoint realization of the Schrödinger operator $\cL + c/m$ with abstract Dirichlet boundary conditions at infinity.
\end{remark}

The subsequent discussion follows \cite[Chapter~7]{KLW21}. Since $Q$ is a Dirichlet form, the induced semigroup  $(e^{-tL})_{t > 0}$ and resolvent $((L + \beta)^{-1})_{\beta > 0}$ are positivity preserving. Hence, they can be extended to linear maps on the cone $C(X)_+ = \{f \in C(X) \mid f \geq 0\}$. More precisely, for $x\in X$ we let   
$$e^{-tL} f(x) := \sup \{e^{-tL}g(x) \mid g \in \ell^2(X,m) \text{ with } 0 \leq g \leq f\} $$
and 
$$(L + \beta)^{-1} f(x) := \sup \{(L + \beta)^{-1} g(x) \mid g \in \ell^2(X,m) \text{ with } 0 \leq g \leq f\}.$$
 These extensions may attain the value $\infty$. Since both resolvent and semigroup are continuous on $\ell^2(X,m)$, it suffices to take the $\sup$ over $g \in C_c(X)$ with $0 \leq g \leq f$. 

Since $Q$ is a Dirichlet form, the semigroup and resolvent are not only positivity preserving but also Markovian. This means that for $0 \leq f \leq 1$ we have
$$0 \leq e^{-tL}f \leq 1 \text{ and } 0 \leq  \beta (L+\beta)^{-1} f \leq 1. $$
In particular, for nonnegative bounded functions the extended resolvent and semigroup only attain finite values.

\begin{definition}[Stochastic completeness]
 Assume $c = 0$. We say that $b$ over $(X,m)$  is {\em stochastically complete} if for all $t >0$ the semigroup $(e^{-tL})_{t>0}$ satisfies the {\em   conservation property} $1 = e^{-tL}1.$
\end{definition}

 If $c \neq 0$,  then the semigroup $(e^{-tL})_{t>0}$  never satisfies the conservation property. This is why the following concept of stochastic completeness at infinity was introduced in \cite{kel12}.

\begin{definition}[Stochastic completeness at infinity]
 We say that $(b,c)$ over $(X,m)$ is {\em stochastically complete at infinity} if for all $t >0$ the semigroup $(e^{-tL})_{t>0})$   satisfies the {\em generalized conservation property}

 $$1 = e^{-tL}1 + \int_0^t e^{-sL}\frac{c}{m} ds.$$
\end{definition}

\begin{remark}
The conservation property is equivalent to $1 = \beta(L  + \beta)^{-1} 1$ for all $\beta > 0$ and the generalized conservation property is equivalent to $1 = \beta(L  + \beta)^{-1} 1 + (L+\beta)^{-1}\frac{c}{m}$ for all $\beta > 0$.
\end{remark}

Below we need the following characterizations of stochastic completeness (at infinity), which can be found in \cite[Theorem~7.2]{KLW21}.

\begin{proposition}[Some characterizations of stochastic completeness at infinity]\label{proposition:stochastic completeness}
 The following assertions are equivalent. 
 \begin{enumerate}[(i)]
  \item $(b,c)$ over $(X,m)$ is stochastically complete at infinity. 
  \item For one/all $\beta > 0$ every $0 \leq g \in \ell^\infty(X)$ with $\cL g + \frac{c}{m} g + \beta g \leq 0$ satisfies $g = 0$.
  \item   If $f \in \cF$ satisfies $f^* := \sup f \in (0,\infty)$ and $  \delta \in (0,f^*)$, then 
  $$\sup_{x \in X_\delta } \left(\cL f(x) + \frac{c(x)}{m(x)}f(x)\right)  \geq 0,$$
  where $X_\delta  = \{x \in X \mid f(x) > f^* - \delta\}$ ('Omori-Yau maximum principle').
  \end{enumerate}

\end{proposition}

\section{A resolvent of non-linear Schrödinger operators with bounded below potential}

In this section we introduce minimal  resolvents of the non-linear Schrödinger operator $\varphi^{-1}\circ \mathcal L + W$. Here  and throughout this section we make the following assumptions.

{\bf Assumptions:} 
\begin{itemize}
 \item $\varphi \colon \IR \to \IR$ is continuous and strictly increasing with $\varphi(0) = 0$.
 \item $W \colon X \to (0,\infty)$ with $W_0 := \inf_{x \in X} W(x) > 0$.
\end{itemize}

Since $\varphi$ is strictly increasing and continuous, its range is a possibly unbounded open interval and $\varphi$ is a bijection onto it. Hence, $\varphi^{-1}(\cL f(x))$ is well-defined as long as $\cL f(x) \in {\rm ran}\, \varphi$. We let 
$$\cF_\varphi = \{f \in \cF \mid \cL f(x) \in {\rm ran}\,\varphi \text{ for all }x \in X\}$$
be the domain of the {\em formal non-linear Schrödinger operator} $\varphi^{-1}\circ \mathcal L + W$.  For $f \in \cF_\varphi$ and $x \in X$ the operator acts by 
$$(\varphi^{-1}\circ \mathcal L + W)(f)(x) = \varphi^{-1}(\cL f(x)) + W(x) f(x).$$
Note that $\cF_\varphi$ is convex.

For nonnegative $f \in C(X)$ we construct a minimal nonnegative function  $g \in \cF_\varphi$ that solves the equation 
$$\varphi^{-1} \left( \mathcal L   g \right)  + W g   = f$$
(if such nonnegative solutions exist).  Using certain convex functionals we first show existence for finitely supported  $f$ and  extend it by using monotonicity properties. We then use these resolvents to characterize when all nonnegative bounded solutions to the semi-linear equation 
$$- \cL g = \varphi(W g)$$
vanish.

\subsection{Dirichlet resolvent of finitely supported functions}

In this subsection we use variational methods to construct the minimal solution to $\varphi^{-1} \left( \mathcal L   g \right)  + W g   = f$ for compactly supported $f$, which we call Dirichlet resolvent.

We denote by $\Phi$ the antiderivative of $2\varphi$ with $\Phi(0) = 0$, i.e.,
$$\Phi \colon \IR \to \IR,\quad \Phi(s) = \int_0^s 2 \varphi(t)dt.$$
Then $\Phi$ is nonnegative and since $\varphi$ is strictly increasing, $\Phi$ is strictly convex.

Let $f \in C_c(X)$ and $\kappa_f = \kappa_f^{\varphi,W} \colon C(X) \to [0,\infty]$,  
$$\kappa_f(u) = \sum_{x \in X} \Phi(f(x) -  W(x)u(x))\frac{m(x)}{W(x)}.$$
We  define the functional $\cE_f = \cE_f^{\varphi,W}$ by $D(\cE_f) = \{u \in \cD_0 \mid \kappa_f(u) < \infty\}$ and
$$\cE_f \colon D(\cE_f) \to [0,\infty),\quad \cE_f(u) = \cQ(u) +   \kappa_f(u). $$

\begin{remark}
Since $f$ has finite support and $\Phi(0) = 0$, we have 
$$\kappa_f(g) = \sum_{x \in X} \Phi(f(x) -  W(x)g(x))\frac{m(x)}{W(x)} < \infty$$
for all $g \in  C_c(X)$, which implies $C_c(X) \subseteq D(\cE_{f})$.  Moreover, if $u \in D(\cE_f)$ and $g \in C_c(X)$, then $u + g \in D(\cE_f)$. 

%
\end{remark}

\begin{lemma}[Pointwise lower semicontinuity]\label{lemma:lower semicontinuity of E}
 If $(u_n)$ is a sequence in $D(\cE_{f})$ with $u_n \to u$ pointwise and $\liminf_{n \to \infty} \cE_{f}(u_n) < \infty$, then $u \in D(\cE_{f})$ and 
 $$\cE_{f}(u) \leq \liminf_{n \to \infty} \cE_{f}(u_n).$$
\end{lemma}
\begin{proof}
 The continuity of $\Phi$ and Fatou's lemma imply 
 $$ \sum_{x \in X}  \Phi(f(x) -  W(x) u(x))\frac{m(x)}{W(x)} \leq \liminf_{n \to \infty} \sum_{x \in X}  \Phi(f(x) -  W(x) u_n(x))\frac{m(x)}{W(x)}.$$
 This observation combined with Lemma~\ref{lemma:pointwise semicontinuity} yields the claim.  
\end{proof}


\begin{lemma}[Approximation in $C_c(X)$]\label{lemma:cc approximation}
 For any $0 \leq u \in D(\cE_f)$ there exists a sequence $(g_n)$ in $C_c(X)$ with $0 \leq g_n \leq u$ such that 
 $$\cE_f(u) = \lim_{n\to \infty} \cE_f(g_n).$$
\end{lemma}
\begin{proof}
According to Lemma~\ref{lemma:bounded approximation} there exists a sequence $(g_n)$ in $C_c(X)$ with  $0 \leq g_n \leq u$, $g_n \to u$ pointwise and $\cQ(g_n - u) \to 0$. Since $\cQ$ is a quadratic form, we also have $\cQ(g_n) \to \cQ(u)$. The inequality $0 \leq g_n \leq u$  yields
$$f - Wu \leq f - W g_n \leq f,$$
which by the definition of $\Phi$ implies 
$$\Phi(f - W g_n) \leq \Phi(f)  + \Phi(f - Wu).$$
Since $f \in C_c(X)$ and $\kappa_f(u) < \infty$, the right side of this inequality is summable with respect to the weight $m/W$. Hence, the pointwise convergence $\Phi(f - W g_n) \to \Phi(f - Wu)$ and Lebesgue's dominated convergence theorem yield $\kappa_f(g_n) \to \kappa_f(u)$. Together with the already established $\cQ(g_n) \to \cQ(u)$ this proves the claim.
%
%
%
\end{proof}

\begin{proposition}[Existence of the Dirichlet resolvent]\label{proposition:definition non-linear resolvent}
There exists a unique function $R f = R^{\varphi,W} f \in D(\cE_{f})$ such that 
$$\cE_{f}(R f) = \inf \{\cE_{f}(u) \mid u \in D(\cE_{f})\} = \inf \{\cE_{f}(g) \mid g \in C_c(X) \}. $$
It satisfies $W_0 \av{R f}_\infty \leq \av{f}_\infty$ and $Rf \in \cF_\varphi$,  and  solves the equation 
$$\varphi^{-1} \left( \mathcal L  R f \right)  + W R f   = f.$$
\end{proposition}
\begin{proof}
The strict convexity of $\Phi$ and the convexity of $\cQ$ imply the strict convexity of $\cE_{f}$. Hence,  if a minimizer $Rf$ exists, it is unique. That taking the infimum over $D(\cE_f)$ and over $C_c(X)$ yields the same result follows from Lemma~\ref{lemma:cc approximation}.

To prove existence we first show that functions in a minimizing sequence can be chosen to satisfy a uniform pointwise bound.  Let $K = \av{f}_\infty/W_0$ and consider the normal contraction $C \colon \IR \to \IR, C(t) = (t \wedge K)\vee(-K)$. The definition of $\Phi$ implies $\Phi(s) \leq \Phi(t)$ if $|s| \leq |t|$. Moreover, if $|t| \leq \av{f}_\infty$, then for all $x \in X$ we have
\begin{align*}
 |t  - W(x) C(s)| &= W(x) |t/W(x) - C(s)| = W(x) |C(t/W(x)) - C(s)| \\
 &\leq   |t - W(x) s|.  
\end{align*}
Combining both inequalities we obtain 
$$\Phi(f(x) - W(x) C(u(x))) \leq \Phi(f(x) - W(x)  u(x))$$
for $u \in C(X)$ and $x \in X$. 

Since $\calQ$ on $\calD_0$ is compatible with normal contractions, see Lemma~\ref{lemma:compatibility with contractions}, we arrive at $C \circ u \in D(\cE_{f})$ and $\cE_{f}(C \circ u) \leq \cE_{f}(u)$ for any $u \in D(\cE_{f})$. 

Let $(u_n)$ be a minimizing sequence for $\cE_{f}$. By the contraction property we proved above we can assume without loss of generality $u_n = C \circ u_n$, i.e., $|u_n| \leq K = \av{f}_\infty/W_0$. With this at hand we can assume without loss of generality that there exists $Rf \in C(X)$ with $u_n \to  R f$ pointwise. Then $\av{R f}_\infty \leq \av{f}_\infty/W_0$ and the lower semicontinuity of $\cE_{f}$ with respect to pointwise convergence yields $ R f \in D(\cE_{f})$ and
$$\cE_{f}( R f) \leq \liminf_{n \to \infty} \cE_{f}(u_n) =  \inf \{\cE_{f}(u) \mid u \in D(\cE_{f})\}.$$

To prove  that the minimizer  satisfies the claimed equation we recall  $Rf + g \in D(\cE_{f})$ for any $g \in C_c(X)$.  For $h \in \IR$ and $x \in X$ this implies 
$$\cE_{f}(Rf + h \delta_x) \geq \cE_{f}(Rf).$$
Here, $\delta_x(x) = 1$ and $\delta_x(y) = 0$ if $x \neq y$.  Comparing the expressions on both sides of this inequality and using that $\calQ$ is a quadratic form we obtain 
\begin{align*}
 &2h \cQ (Rf,\delta_x) + h^2 \calQ(\delta_x)\\
 &\geq \frac{m(x)}{W(x)}  \left(\Phi(f(x) - W(x) Rf(x)) - \Phi(f(x) - W(x) (Rf(x) + h)   \right)
%
%
\end{align*}
Dividing by $2h$, letting $h \to 0\pm$ and using $\Phi' = 2\varphi$ yields 
$$\cQ (Rf,\delta_x) = m(x) \varphi(f(x) - W(x) R f(x)). $$
Dividing by $m(x)$ and using Green's formula Lemma~\ref{lemma:greens formula}, we infer $\cL Rf (x) =  \varphi(f(x) - W(x) Rf(x))$, which shows the claim. 
\end{proof}

In order to establish further properties of $Rf$ we approximate it via resolvents on finite sets, where we can employ a comparison principle (see below). To this end, for $U \subseteq X$ and $f \in C_c(X)$ we define the convex functional $\cE_{f,U} = \cE_{f,U}^{\varphi,W}$ as the restriction of $\cE_f$ to $D(\cE_{f,U}) = \{u \in D(\cE_f) \mid {\rm supp}\, u \subseteq U\}.$ The same arguments as in the proof of Proposition~\ref{proposition:definition non-linear resolvent} show that it has a unique minimizer $R_U f = R_U^{\varphi,W} f \in D(\cE_{f,U})$, which satisfies $W_0 \av{R_U f}_\infty \leq \av{f}_\infty$ and $R_U f \in \cF_\varphi$, and solves the non-local Dirichlet problem
$$\begin{cases}
   \varphi^{-1} \left( \mathcal L  R_U f \right)  + W R_U f   = f &\text{on }U\\
   R_U f = 0 &\text{on } X \setminus U
  \end{cases}.
$$
\begin{definition}[Dirichlet resolvent]
We call $Rf$ (respectively $R_U f$) the {\em Dirichlet resolvent} (respectively the {\em Dirichlet resolvent on $U$}) of $f$ with respect to the operator $\varphi^{-1} \circ \mathcal L + W$. 
\end{definition}

\begin{remark}[Neumann resolvent]
 If we only wanted to construct solutions $u$ to the non-linear equation $\varphi^{-1}(\mathcal L u) + W u =f$, instead of  $\cE_f$, which is defined on a subset of $\cD_0$, we could have considered the functional 
 $$\widetilde{\cE}_f(u) = \cQ(u) + \kappa_f(u),   $$
 on the domain $D(\widetilde{\cE}_f) = \{u \in \cD \mid \kappa_f(u) < \infty\}$. It also has a unique minimizer $\widetilde{R}f$ that solves the mentioned  equation. The  function $\widetilde R f$ is the {\em Neumann resolvent} of $f$ with respect to the operator $\varphi^{-1} \circ \mathcal L + W$. The names Dirichlet resolvent and Neumann resolvent come from the observation that they correspond to solutions with abstract 'Dirichlet boundary conditions' and 'Neumann boundary conditions', respectively. This analogy can be best seen for the Dirichlet resolvent on a subset $U$. 
 
 We use the Dirichlet resolvent because it can be approximated by resolvents on finite subsets and turns out to be the minimal solution. This is discussed next. 
\end{remark}

%

%

%
%

The main tool for studying properties of the resolvent is the following comparison principle.  

\begin{lemma}[Comparison principle - I] \label{lemma:comparison principle}
 Let $u,v \in  \cF_\varphi$ and assume that for some $U \subseteq X$  they satisfy the following conditions: 
  \begin{enumerate}[(a)]
   \item   $\phi\inv(\cL u) + W u \geq \phi\inv(\cL v) + Wv$ on $U$.
   \item $u-v$ attains a minimum on $U$.
   \item   $u \geq v$ on $X \setminus U$.
  \end{enumerate}
 Then $u \geq v$ on $X$.
\end{lemma}

\begin{proof}
 Let $g:= u-v$ attain its minimum on $U$ at $x \in U$.  If $g(x) < 0$, then  assumption (a)  and $W(x) > 0$  imply
  $$ \phi\inv(\cL u(x)) - \phi\inv(\cL v(x)) \geq  - W(x) g(x) > 0.$$
 The strict monotonicity of $\phi$ yields 
 $$\frac{1}{m(x)} \sum_{y \in X} b(x,y) (g(x) - g(y)) = \cL g(x) =  \cL u(x) - \cL v(x)  > 0.$$
 This implies the existence of $y \sim x$ such that $g(y) < g(x) < 0$. Since $g$ attains its minimum on $U$ in $x$, we infer $y \in X \setminus U$. But this implies $g(y) \geq 0$ by assumption (c), which is a contradiction. Hence, $u - v = g \geq 0$ on $U$. 
\end{proof}

As a corollary we also note the following version of the comparison principle, which compares different nonlinearities and potentials, but only holds under a further nonnegativity assumption. 

\begin{corollary}[Comparison principle - II] \label{coro:comparison principle}
 Let $\psi \colon \IR \to \IR$ be strictly increasing, continuous with $\psi (0) = 0$ and $\psi(t) \leq \varphi(t)$ for $t \geq 0$. Let $V \colon X \to (0,\infty)$ with $V \geq W$. Let $u \in \cF_\varphi$ and $v \in  \cF_\psi$ and assume that for some $U \subseteq X$ they satisfy the following conditions: 
  \begin{enumerate}[(a)]
   \item   $\phi\inv(\cL u) + W u \geq \psi\inv(\cL v) + Vv$ on $U$.
   \item $u-v$ attains a minimum on $U$.
   \item   $u \geq v$ on $X \setminus U$.
   \item  $\cL v \geq 0$ on $U$ and $v \geq 0$ on $U$.
  \end{enumerate}
Then $u \geq v$ on $X$.
\end{corollary}

\begin{proof}
 Since  $\psi,\varphi$ are continuous and strictly increasing, their ranges are open intervals. Hence, the inequality $\psi(t) \leq \varphi(t)$, $t \geq 0$, yields 
 $$\ran \psi \cap [0,\infty) \subseteq \ran \varphi \cap [0,\infty) $$
 and $\psi\inv (s) \geq \varphi \inv (s)$ for $s \geq 0$ with $s \in \ran \psi$. Using these observations and the assumptions $\cL v \geq 0$ and $v \in \cF_\psi$, shows $v \in \cF_\varphi$. Moreover, using (a) and $\cL v \geq 0$, $v \geq 0$, $V \geq W$, we obtain 
 $$\phi\inv(\cL u) + W u \geq \psi\inv(\cL v) + Vv \geq \phi\inv(\cL v) + W v. $$
 With this at hand the statement follows with the assumptions (b) and (c) and the previous lemma. 
\end{proof}

\begin{lemma}[Domain monotonicity]
 Let $U,V \subseteq X$ be finite with $U \subseteq V$. Furthermore, let $f,g \in C_c(X)$ with $0 \leq f \leq g$. Then 
 $$R_U f \leq R_V g. $$
 In particular, $R_V g \geq 0$. 
\end{lemma}
\begin{proof}
We first show $R_V g \geq 0$.  By definition $R_V g = 0$ on $X \setminus V$ and $\varphi^{-1} \left( \mathcal L  R_V g \right)  + W R_V g   = g \geq 0$ on $V$. Moreover, $R_V g$ attains  its minimum on the finite set $V$. Hence, the comparison principle Lemma~\ref{lemma:comparison principle} applied to $u = R_Vg$ and $v = 0$ yields $R_V g \geq 0$. 

The other inequality also follows from the comparison principle. More precisely, on $U$ we have 
$$\varphi^{-1} \left( \mathcal L  R_V g \right)  + W R_V g   = g \geq f = \varphi^{-1} \left( \mathcal L  R_U f \right)  + W R_U f $$
and on $X \setminus U$ we have
$$R_V g \geq 0 = R_U f.$$
Since $R_V g - R_U f$ attains its minimum on the finite set $U$, we infer $R_U f \leq R_V g$ from the comparison principle Lemma~\ref{lemma:comparison principle}.
\end{proof}

\begin{lemma}[Approximation via finite subsets - I]\label{lemma:approximation via finite subsets I}
Let $0 \leq f \in C_c(X)$ and let $(K_n)$ be finite subsets in $X$ with $K_n \nearrow X$.  Then $R_{K_n} f \nearrow R f$ pointwise. 
\end{lemma}
\begin{proof}
 The domain monotonicity implies that $R_{K_n} f$ is increasing. Since $W_0 \av{R_{K_n} f}_\infty \leq \av{f}_\infty$, the limit 
 $$F := \lim_{n \to \infty} R_{K_n} f$$
 exists pointwise. We  show that it is the minimizer of $\cE_f$.  
 
 The domains $D(\cE_{f,K_n})$ are increasing. Therefore,  $\cE_f(R_{K_n} f) = \cE_{f,K_n}(R_{K_n} f)$ is decreasing and bounded by $\cE_{f,K_1}(R_{K_1} f)$.  The pointwise lower semicontinuity of $\cE_f$, Lemma~\ref{lemma:lower semicontinuity of E},  yields $F \in D(\cE_f)$ and 
 $$\cE_f(F) \leq \liminf_{n\to \infty} \cE_f(R_{K_n} f) = \liminf_{n\to \infty} \cE_{f,K_n}(R_{K_n} f).  $$
 For $g \in C_c(X)$ there exists $N \in \IN$ such that ${\rm supp}\, g \subseteq K_n$ for each $n \geq N$. Since $R_{K_n}f$ is the minimizer of $\cE_{f,K_n}$, this implies 
 $$\cE_{f,K_n}(R_{K_n}f) \leq \cE_{f,K_n}(g) = \cE_f (g)$$
 for $n \geq N$. Combined with the previous inequality, we infer $\cE_f(F) \leq \cE_f(g)$ for all $g \in C_c(X)$. According to Proposition~\ref{proposition:definition non-linear resolvent}, this implies $F = Rf$.  
\end{proof}

\begin{corollary}[Monotonicity of the resolvent]\label{coro:monotonicity resolvent}
 Let $f,g \in C_c(X)$ with $0 \leq f \leq g$. Then $0 \leq Rf \leq Rg$.
\end{corollary}
\begin{proof}
 This follows from combining the previous two lemmas. 
\end{proof}

We established the monotonicity of the resolvent as a consequence of the comparison principle Lemma~\ref{lemma:comparison principle} and the approximation via finite sets. Using the  same argument with the slightly more general comparison principle Corollary~\ref{coro:comparison principle}, we obtain the following result on domination of the resolvents.

\begin{proposition}[Domination]\label{proposition:domination prelim}
Let $\psi \colon \IR \to \IR$ be strictly increasing, continuous with $\psi (0) = 0$ and $\psi(t) \leq \varphi(t)$ for $t \geq 0$. Let $V \colon X \to (0,\infty)$ with $V \geq W$. For all $0 \leq f \in C_c(X)$ we have 
$$0 \leq R^{\psi,V} f \leq R^{\varphi,W}f.$$
\end{proposition}

\subsection{The extended Dirichlet resolvent}

The monotonicity of the Dirichlet resolvent on finitely supported functions allows us to extend it to all positive functions. In this subsection we show that this extension yields minimal solutions to the equation $\varphi^{-1}(\cL g) + W g = f$.

\begin{definition}[Extended Dirichlet resolvent]
 For $0 \leq f \in C(X)$ we define the {\em extended Dirichlet resolvent} by
$$Rf\colon X \to [0,\infty], \quad Rf(x) = \sup \{Rg(x) \mid g \in C_c(X) \text{ with } 0 \leq g \leq f\}. $$
\end{definition}

As a first consequence of this definition we note that the domination we observed at the end of the last section extends to the extended resolvent.

\begin{proposition}[Domination] \label{proposition:domination}
Let $\psi \colon \IR \to \IR$ be strictly increasing, continuous with $\psi (0) = 0$ and $\psi(t) \leq \varphi(t)$ for $t \geq 0$. Let $V \colon X \to (0,\infty)$ with $V \geq W$. For all $0 \leq f \in C(X)$ we have 
$$0 \leq R^{\psi,V} f \leq R^{\varphi,W} f.$$
\end{proposition}
\begin{proof}
 This follows directly from Proposition~\ref{proposition:domination prelim} and the definition of the extended resolvent. 
\end{proof}

\begin{lemma}[Approximation via finite subsets - II]\label{lemma:approximation via finite subsets II}
Let $0 \leq f \in C(X)$ and let $(K_n)$ be finite subsets in $X$ with $K_n \nearrow X$.  Then $R_{K_n} (f1_{K_n})  \nearrow R f$ pointwise. 
\end{lemma}
\begin{proof}
The approximation via finite subsets of the resolvent on finitely supported functions, Lemma~\ref{lemma:approximation via finite subsets I}, and the definition of the extended resolvent yield 
$$R_{K_n} (f1_{K_n}) \leq R (f 1_{K_n}) \leq Rf.$$
Moreover, domain monotonicity implies that $R_{K_n} (f1_{K_n})$ is increasing in $n$. 

Let $x \in X$. Case~1: $Rf(x) < \infty$. Let $\varepsilon > 0$ and choose $w \in C_c(X)$ with $0 \leq w \leq f$ and $Rf(x) \leq Rw(x) + \varepsilon$. Using Lemma~\ref{lemma:approximation via finite subsets I} and $w \leq f 1_{K_n}$ for large enough $n$,  we infer 
$$Rf(x) \leq Rw(x) + \varepsilon = \lim_{n\to \infty}  R_{K_n} w(x) + \varepsilon \leq \lim_{n\to \infty}  R_{K_n} (f1_{K_n})(x) + \varepsilon.   $$

Case 2: $Rf(x) = \infty$: Let $C \geq 0$ and choose $0 \leq w \leq f$ with $Rw(x) \geq C$. For $n$ large enough we have $f1_{K_n} \geq w$ and infer 
$$C \leq Rw(x)  =  \lim_{n\to \infty} R_{K_n} w(x) \leq \lim_{n\to \infty} R_{K_n} (f1_{K_n})(x), $$
where we use Lemma~\ref{lemma:approximation via finite subsets I} for the equality.
\end{proof}

\begin{theorem}[Minimality of the extended Dirichlet resolvent]\label{theorem:minimality resolvent}
 Let $0 \leq f \in C(X)$. 
 \begin{enumerate}[(a)]
  \item $Rf(x) < \infty$ for all $x \in X$ implies $Rf \in \cF_\varphi$ and 
  $$\varphi^{-1}(\cL Rf) + W Rf  = f.$$
  \item $Rf(x) < \infty$ for all $x \in X$ if and only if there exists $0 \leq g \in \cF_\varphi$ with 
  $$\varphi^{-1}(\cL g) + W g \geq f. $$
  In this case,  $g \geq Rf$.
 \end{enumerate}
\end{theorem}
\begin{proof}
 (a): Choose a sequence of finite sets $(K_n)$ with $K_n \nearrow X$ and let $g_n := R_{K_n} (f1_{K_n})$. The definition of $R_{K_n}$ yields 
 $$\cL g_n = \varphi(f - W g_n) \text{ on } K_n. $$
 Spelled out, for each $x \in K_n$ this reads
 $$\frac{\deg(x)}{m(x)} g_n(x) - \frac{1}{m(x)}\sum_{y \in X} b(x,y) g_n(y) = \varphi(f(x) - W(x) g_n(x)).$$
Lemma~\ref{lemma:approximation via finite subsets II} implies $g_n \nearrow Rf$. The monotone convergence theorem yields that the left side of the previous equation converges and the right side of the equation converges because $\varphi$ is continuous. Hence, we obtain 
$$\frac{\deg(x)}{m(x)} Rf(x) - \frac{1}{m(x)}\sum_{y \in X} b(x,y) Rf(y) = \varphi(f(x) - W(x)Rf(x))$$
for all $x \in X$. In particular, this shows $\sum_{y \in X} b(x,y)Rf(y) < \infty$  for all $x \in X$ (i.e. $Rf \in \cF$) and   $\cL Rf = \varphi(f - W Rf)$. 
 
(b): If $Rf(x) < \infty$ for all $x \in X$, then using (a) we can choose $g = Rf$. 

Now suppose $0 \leq g \in \cF_\varphi$ with $\varphi^{-1}(\cL g) + W g \geq f.$ Let $g_n = R_{K_n} (f1_{K_n})$ be as in the proof of (a). Then 
$$\varphi^{-1}(\cL g) + W g \geq f = \varphi^{-1}(\cL g_n) + W g_n \text{ on } K_n. $$
Moreover, $g \geq 0 = g_n$ on $X \setminus K_n$ and $g - g_n$ attains a minimum on the finite set $K_n$. Hence, the comparison principle Lemma~\ref{lemma:comparison principle} yields $g \geq g_n$.  The approximation of Lemma~\ref{lemma:approximation via finite subsets II} yields $g \geq Rf$.
\end{proof}

\begin{corollary} \label{corollary:inequality for potential} For any $\alpha \geq 0$ we have  $R(\alpha W) \leq \alpha$.  
\end{corollary}
\begin{proof}
 The constant function $f = \alpha 1_X$ satisfies 
 $$\varphi^{-1}(\cL f) + W f = \alpha W.$$
 With this at hand the previous theorem shows $R(\alpha W) \leq f = \alpha 1_X$. 
\end{proof}
\begin{remark}
\begin{enumerate}[(a)]
 \item Recall the situation of Subsection~\ref{subsection:stochastic completeness}. In this case, if we let $\varphi = {\rm id}_\IR$ and $W = \beta + \frac{c}{m}$, then for $0 \leq f \in C(X)$ we have $Rf = (L + \beta)^{-1}f$. Here, $(L+\beta)^{-1}f$ is the extended resolvent discussed in Subsection~\ref{subsection:stochastic completeness}. In particular, we obtain
 $$R (\beta + \frac{c}{m}) = \beta (L+\beta)^{-1} 1 + (L+\beta)^{-1} \frac{c}{m}. $$
 This shows that  $R(\alpha W)$ is the non-linear analogue of the quantity appearing in the generalized conservation property. In particular, the previous theorem and its corollary are non-linear versions of \cite[Theorem~11 and Proposition~6.1]{kel12}.
 \item For general $\varphi$ the resolvent $R$ is not homogenous. It is therefore quite remarkable that the  inequality $R(\alpha W) \leq \alpha$ holds for all $\alpha \geq 0$.  
\end{enumerate}
\end{remark}

\section{Stochastic completeness at infinity and uniqueness of bounded (super)harmonic functions}

In this section we characterize triviality of bounded  solutions (nonpositive bounded supersolutions) to the equation 
$$\varphi^{-1}(\cL g) + Wg = 0$$
with the help of the non-linear resolvent introduced in the previous sections. The most important observation is that resolvents yield minimal bounded below (super)solutions to it, which we discuss in the following lemma.   

%
%
%

\begin{lemma}[Minimal (super)solution]\label{lemma:minimality of resolvent}
   Let $\alpha \geq 0$. The function $w = R(\alpha W) - \alpha$ satisfies $-\alpha \leq w \leq 0$, $w \in \cF_\varphi$ and $\varphi \inv (\cL w) + Ww = 0$.  If $h \in \cF_\varphi$ satisfies  $\varphi^{-1} (\cL h) + W h \geq 0$ and $h \geq -\alpha$, then $w \leq h$. 
\end{lemma}

\begin{proof}
By Corollary~\ref{corollary:inequality for potential} $w$ is nonpositive and by Corollary~\ref{coro:monotonicity resolvent} it is bounded from below by $-\alpha$. Moreover, Theorem~\ref{theorem:minimality resolvent} shows $w \in \cF_\varphi$ and
 $$\varphi^{-1}(\cL w) + W w = \varphi^{-1}(\cL R(\alpha W)) + W R(\alpha W) - \alpha W  = \alpha W - \alpha W = 0. $$

 Let $g = \alpha + h \geq 0$. Then  $g \in \cF_\varphi$ with 
$$\varphi^{-1}(\cL g) + W g = \varphi^{-1} (\cL h) + W h + \alpha W \geq \alpha W. $$
With this at hand Theorem~\ref{theorem:minimality resolvent} yields $R(\alpha W) \leq g = \alpha  + h$.
\end{proof}

With this at hand we obtain the following characterization of an $\ell^\infty$-Liouville type theorem for the operator $\varphi^{-1}\circ \cL  + W$.

\begin{theorem}[Uniqueness of bounded supersolutions]\label{theorem:uniqueness of bounded subsolutions}
The following assertions are equivalent. 
 \begin{enumerate}[(i)]
  \item $R(\alpha W) = \alpha$ for all $\alpha \geq 0$.
 
   \item Any $0 \geq g \in \ell^\infty(X) \cap \cF_\varphi$ with $\varphi^{-1}(\cL g) + Wg = 0$ satisfies $g = 0$.
  
   \item Any $0 \geq g \in \ell^\infty(X) \cap \cF_\varphi$ with $\varphi^{-1}(\cL g) + Wg \geq 0$ satisfies $g = 0$.
  
  \item Any $g \in \ell^\infty(X) \cap \cF_\varphi$ with $\varphi^{-1}(\cL g) + Wg \geq 0$ satisfies $g \geq 0$.
  
  \item Any $0 \leq f \in \ell^\infty(X)$ with $-\cL f = \varphi(Wf)$ satisfies $f = 0$. 
 \end{enumerate}
%
 %
\end{theorem}
\begin{proof}

(i) $\Rightarrow$ (iv): Let $\alpha = \av{g}_\infty$. Then $g \geq -\alpha$ and Lemma~\ref{lemma:minimality of resolvent} yields $R(\alpha W) - \alpha  \leq g$. Since by assumption $R(\alpha W) = \alpha$, this leads to $g \geq 0$.

(iv) $\Rightarrow$ (iii) $\Rightarrow$ (ii): This is trivial. 

(ii) $\Rightarrow$ (i): Let $\alpha \geq 0$ and consider the function $w =  R(\alpha W) - \alpha$. As seen in Lemma~\ref{lemma:minimality of resolvent} it satisfies $0 \geq w \in \ell^\infty(X) \cap \cF_\varphi$ and $\varphi \inv (\cL w) + Ww = 0$.  Hence, (ii) implies $w = 0$.  
 
(ii) $\Leftrightarrow$ (v): It is readily verified that $ g \in  \cF_\varphi$ and $\varphi^{-1}(\cL g) + Wg = 0$ if and only if $g \in \cF$ and $\cL g = \varphi(-Wg)$. Since also $\ell^\infty(X) \subseteq \cF$, the equivalence of (ii) and (v) follows immediately. 
%
%
%
%
%
%
%
%
%
\end{proof}

\begin{remark} \begin{enumerate}[(a)]

\item It is quite remarkable that the conditions (ii), (iii) and (v) in the previous theorem  only depend on $\varphi|_{[0,\infty)}$ (and hence also the equivalent conditions (i) and (iv)).  The reason is that $R(\alpha W)$, which yields minimal supersolutions, does only depend on $\varphi|_{[0,\infty)}$ (this can be seen by carefully following the construction of the resolvent). Note however, that for general $f \geq 0$ the resolvent $R f$ may depend on $\varphi|_{(-\infty,0)}$.

\item  For $W = 1$  the equation $-\cL f = \varphi(f)$ in (v) is a discrete version of the semi-linear elliptic equation considered in \cite{gri20,gri23}. We discuss its connection to stochastic completeness in the next section.   

\item   Assume that $\varphi \colon \IR \to \IR$ is surjective (i.e. $\lim_{t \to \pm \infty} \varphi(t) = \pm \infty$). Then $\cF_\varphi = \cF$.  Since also $\ell^\infty(X) \subseteq \cF$, in this case, the previous theorem simply deals with bounded (super)solutions of  $\varphi^{-1}(\cL g) + Wg = 0$.

                \item  As already discussed in the remark following Theorem~\ref{theorem:minimality resolvent},  if we let $\varphi = {\rm id}_\IR$ and $W = \beta + \frac{c}{m}$, then   
                $$R (\beta + \frac{c}{m}) = \beta (L+\beta)^{-1} 1 + (L+\beta)^{-1} \frac{c}{m}. $$
                Hence, for fixed $\beta > 0$ the identity $R (\beta + \frac{c}{m}) = 1$ is equivalent to $\beta (L+\beta)^{-1} 1 + (L+\beta)^{-1} \frac{c}{m} = 1$. By the resolvent identity this in turn is equivalent to $\beta' (L+\beta')^{-1} 1 + (L+\beta')^{-1} \frac{c}{m} = 1$ for all $\beta' > 0$, which is the generalized conservation property discussed in Subsection~\ref{subsection:stochastic completeness}. 
                
                These observations show that (iv) is a non-linear version of the generalized conservation property and hence of stochastic completeness at infinity. For general $\varphi$ and $W$ the resolvent is non-linear. Hence, it is not sufficient to require $R(W) = 1$ (as the previous computations in the linear case suggest) but we need $R(\alpha W) = \alpha$ for all $\alpha \geq 0$.
               \end{enumerate}
\end{remark}

Based on these observations we adopt the definition of stochastic completeness at infinity to the non-linear setting.

\begin{definition}[Stochastic completeness at infinity - non-linear version] \label{definition:nonlinear stochastic completeness}
 We say that a graph $b$ over  $(X,m)$ with $W$ and $\varphi$ as above is {\em stochastically complete at infinity} if it satisfies the {\em non-linear generalized conservation property}
 $$R^{\varphi,W}(\alpha W) = \alpha, \text{ all } \alpha > 0.$$
\end{definition}

For later purposes we note a stability property of stochastic completeness at infinity.

\begin{proposition}[Stability of stochastic completeness at infinity]\label{proposition:stability of stochastic completeness at infinity}
Let $\psi \colon \IR \to \IR$ be strictly increasing, continuous with $\psi (0) = 0$ and $\psi(t) \leq \varphi(t)$ for $t \geq 0$. Let $V \colon X \to (0,\infty)$ with $\inf_{x \in X} V(x) > 0$ and $W \geq V$. If $b$ over $(X,m)$ with $V$ and $\psi$ is stochastically complete at infinity, then it is stochastically complete at infinity with $W$ and $\varphi$. 
\end{proposition}
\begin{proof}
We show the statement for the case $V = W$ and for the case $\varphi = \psi$. The general statement follows after applying one case first and then the other.

Case 1: $W = V$. For $\alpha > 0$ stochastic completeness at infinity with respect to $\psi, W$ and domination, see Proposition~\ref{proposition:domination}, yields 
$$\alpha = R^{\psi,W}(\alpha W) \leq R^{\varphi,W}(\alpha W) \leq \alpha.$$
This implies $R^{\varphi,W}(\alpha W)  = \alpha$, which is one of our characterizations for stochastic completeness at infinity for $\varphi$,$W$. 

Case 2: $\varphi = \psi$. Let $0 \geq g \in \ell^\infty(X) \cap \cF_\varphi$ with $\varphi^{-1}(\cL g) + W g \geq 0$. Since $V \leq W$ and $g \leq 0$, we obtain 
$$\varphi^{-1}(\cL g) + V g \geq \varphi^{-1}(\cL g) + W g \geq 0.$$
Stochastic completeness at infinity with respect to $\varphi, V$ yields $g = 0$, and we obtain the statement. 
\end{proof}

We finish this section by giving one geometric criterion for this property. 

\begin{proposition}\label{prop:geometric criterion}
Assume that for all $0 < \alpha \leq  1$ and all infinite paths $(x_n)$ in $X$ we have 
$$\sum_{n = 1}^\infty \frac{m(x_n)\varphi(\alpha W(x_n))}{\deg (x_n)} = \infty.$$
Then $b$ over $(X,m)$ with $W$ and $\varphi$ is stochastically complete at infinity. In particular, this holds if $\deg/m$ is bounded. 
\end{proposition}
\begin{proof}
Let $0 \geq g \in \ell^\infty(X) \cap \cF_\varphi$ with $\varphi^{-1}(\cL g) + Wg \geq 0$ be given and assume $g \neq 0$. Moreover, we let $(\rho_n)$ be a sequence with $ \rho_n > 1$, $n \in \IN$, with 
$$ \beta := \prod_{n = 1}^\infty \frac{1}{\rho_n} > 0. $$
By our assumption we find $x_0 \in X$ with $g(x_0) < 0$. We let $f = -g$ and use the boundedness of $f$ to inductively choose a sequence $(x_n)$ in $X$ with $x_{n+1} \sim x_n$ and 
$$\rho_{n+1} f(x_{n+1}) \geq \sup \{f(y) \mid y \sim x_n\}. $$

The inequality $\varphi^{-1}(\cL g) + Wg \geq 0$ leads to $\cL f + \varphi(W f) \leq 0$. Using the definition of $\cL$, for each $n \geq 0$ we infer
\begin{align*}
 \deg(x_n) f(x_n) + m(x_n)\varphi(W(x_n) f(x_n)) &\leq \sum_{y \in X}b(x_n,y)f(y) \\
 &\leq \deg(x_n) \rho_{n+1} f(x_{n+1}).
\end{align*}
This leads to 
$$f(x_{n+1}) \geq \frac{1}{\rho_{n+1}} \left(f(x_n) + \frac{m(x_n)\varphi(W(x_n) f(x_n))}{\deg(x_n)} \right),$$
which immediately yields $f(x_n) \geq \beta f(x_0) > 0$ for all $n \in \IN$.

Iterating the inequality and using $f(x_n) \geq  \beta f(x_0)$, we infer 
$$f(x_{n+1}) \geq \beta \sum_{k=0}^{n} \frac{m(x_k)\varphi(W(x_k) \beta f(x_0))}{\deg(x_k)}. $$
Our assumptions imply that the sum on the right side diverges, as $n \to \infty$ (use that $\varphi(W(x_k) \beta f(x_0)) \geq \varphi(W(x_k))$ if $\beta f(x_0) > 1$). This is a contradiction to $f$ being bounded.

The 'In particular'-part: If for some $C > 0$ we have $\deg/m \leq C$, then 
$$\frac{m(x)\varphi(\alpha W(x))}{\deg (x)} \geq \frac{\varphi(\alpha W_0)}{C}>0,$$
where $W_0 = \inf_{x \in X} W(x) > 0$. Hence, the summability criterion immediately yields the statement. 
\end{proof}

\begin{remark} In the linear case $\varphi = {\rm id}_\IR$ a similar criterion is used to show that given a weighted graph $b$ over  $(X,m)$ one can always find $W$ (respectively $c$ in the notation of Subsection~\ref{subsection:stochastic completeness}) making it stochastically complete at infinity, see  \cite[Theorem~2]{kel12}. For example, $W = \deg/m + 1$ does the job. In the semi-linear case of general $\varphi$ however, this is no longer possible. For example, if $\varphi$ is bounded, the sum in the proposition may just always be finite. In contrast, under some mild growth conditions on $\varphi$, we can establish the existence of $W$ making the graph with $\varphi$ stochastically complete at infinity. This is discussed in the next corollary.  
\end{remark}

\begin{corollary}\label{coro:large potential}
Assume that a weighted graph $b$ over $(X,m)$ is given.  Assume further $\lim_{t \to \infty} \varphi(t) = \infty$ and that for each $0 < \alpha \leq 1$ there exists $C_\alpha > 0$ such that $\varphi(\alpha t) \geq C_\alpha \varphi(t)$ for all $t \geq 0$. Then there exists $W \colon X \to (0,\infty)$ with $\inf_{x \in X} W(x) > 0$ such that $b$ over $(X,m)$ with $\varphi$ and $W$ is stochastically complete at infinity. 
\end{corollary}
\begin{proof}
Using that $\varphi$ tends to $\infty$, as $t \to \infty$, we may choose $W$ such that $ m (\varphi \circ W) \geq \deg$ and $W \geq 1$.  For example, $W = \varphi^{-1} \circ (\deg /m) + 1$ does the job. Using our assumptions on $\varphi$ we infer
$$\frac{m(x) \varphi(\alpha W(x))}{\deg(x)} \geq C_\alpha \frac{m(x) \varphi(W(x))}{\deg(x)} \geq C_\alpha. $$
Hence, the previous proposition yields the claim. 
\end{proof}

\begin{remark}
The assumption of the corollary is satisfied if $\varphi|_{[0,\infty)}$ is concave. For example, it can be applied if $\varphi(t) = \log (1 + t)$ for $t \geq 0$. It is also satisfied for power functions $\varphi(t) = t^p$, $t \geq 0$, where $0 < p < \infty$. 
\end{remark}

\section{Stochastic completeness and semi-linear Liouville properties}

In this section we  apply the theory developed in the previous section to characterize stochastic completeness by semi-linear Liouville properties.  We start with some notation. We let $\mathcal I$ be the set of all strictly increasing continuous functions $\varphi \colon \IR \to \IR$ with $\varphi(0) = 0$. 
Moreover, we let $\mathcal A =\{W \colon X \to [0,\infty) \mid \inf_{x \in X} W(x) > 0\}$ and recall that $\cF_\varphi = \{f \in \cF \mid \cL f(x) \in {\rm ran}\,\varphi \text{ for all }x \in X\}$.


\begin{theorem}[Characterization of stochastic completeness] \label{theorem:stochastic completeness}
The following assertions are equivalent. 
\begin{enumerate}[(i)]
 \item The graph $b$ over $(X,m)$ is stochastically complete. 
 
 \item For one $W \in \mathcal A \cap \ell^\infty(X)$ (any $W\in \mathcal A$) and one $\varphi \in \mathcal I$ (any $\varphi \in \mathcal I$) we have 
 $$ R^{\varphi, W} (\beta W) = \beta, \quad \text{ all } \beta > 0.$$
  \item For one $W \in \mathcal A \cap \ell^\infty(X)$ (any $W\in \mathcal A$) and one $\varphi \in \mathcal \mathcal I$ (any $\varphi \in \mathcal I$) every $0 \geq g \in \ell^\infty(X) \cap \cF_\varphi$ with 
 $\varphi^{-1}(\cL g) + Wg = 0$
 satisfies $g = 0$.
 
  \item For one $W \in \mathcal A \cap \ell^\infty(X)$ (any $W\in \mathcal A$) and one $\varphi \in \mathcal I$ (any $\varphi \in \mathcal I$) every $g \in \ell^\infty(X) \cap \cF_\varphi$ with  $\varphi^{-1}(\cL g) + Wg  \geq 0$ satisfies $g \geq 0$.

 \item For one $W \in \mathcal A \cap \ell^\infty(X)$ (any $W\in \mathcal A$) and one $\varphi \in \mathcal I$ (any $\varphi \in \mathcal I$) every $0\geq g \in \ell^\infty(X) \cap \cF_\varphi$ with 
 $\varphi^{-1}(\cL g) + Wg \geq 0$
 satisfies $g = 0$. 
 
 \item For one $W \in \mathcal A \cap \ell^\infty(X)$ (any $W\in \mathcal A$) and one $\varphi \in \mathcal I$ (any $\varphi \in \mathcal I$) every $0 \leq f \in \ell^\infty(X) $ with 
 $ - \cL f = \varphi(Wf)$
satisfies $f = 0$. 

\end{enumerate}
%
%
%
%
%

\end{theorem}
\begin{proof} In this proof we use the following notation:  If there are 'for one/any'-statements in the assertions,  we attach 'one' or 'any' to the number of the assertion to specify them in the order of appearance. For example '(v) one any' stands for the assertion 'For one $W \in \mathcal A \cap \ell^\infty(X)$ and any $\varphi \in \mathcal I$ every $0\geq g \in \ell^\infty(X) \cap \cF_\varphi$ with 
 $\varphi^{-1}(\cL g) + Wg \geq 0$
 satisfies $g = 0$'. 

Obviously the 'for any'-statements imply the 'for one'-statements and so we only have to prove the opposite implications. 

For fixed $A,B \in \{\text{one,any}\}$, Theorem~\ref{theorem:uniqueness of bounded subsolutions} shows the equivalence of (ii)$A B$, (iii) $A B$, (iv)$A B$, (v)$A B$ and (vi)$A B$. Moreover, it shows that '(iv) any any implies (iii) one one' and  '(iii) one one implies (iv) one one'.


Hence, in order to establish the claimed equivalences it suffices to show '(i) implies (v) any any' and '(v) one one implies (i)'.

(i) implies (v) any any: Let $W \in \mathcal A$,  $\varphi \in \mathcal I$ and let $0 \geq g \in \ell^\infty(X) \cap \cF_\varphi$ with $\varphi^{-1}(\cL g) + W g \geq 0$. Below we apply the Omori-Yau maximum principle of Proposition~\ref{proposition:stochastic completeness} to $f = -g$ in order to conclude $g = 0$. 

Let $\alpha = \inf_{x \in X} W(x) > 0$. Since $f$ is nonnegative, we have $\alpha f \leq Wf$, which yields (after rearranging the inequality for $g$)
$$\cL f + \varphi(\alpha f) \leq \cL f + \varphi(W f) \leq 0. $$

Assume that  $f^* = \sup_{x \in X} f(x) > 0$. We let  $\varepsilon >0$  and choose $\delta > 0$ such that $|t- \alpha f^*| < \delta$ implies $|\varphi(t) - \varphi(\alpha f^*)| < \varepsilon$. The Omori-Yau maximum principle implies the existence of $\xi \in X$ with $f(\xi) > f^* - \delta/\alpha$ and $\mathcal \cL f(\xi) > -\varepsilon$. Then $|\alpha f(\xi) - \alpha f^*| < \delta$ and our choices of the parameters yield 
$$\varphi(\alpha f^*) \leq  \varphi(\alpha f(\xi)) + \varepsilon \leq  - \cL f(\xi) + \varepsilon  < 2 \varepsilon.  $$
Since $\varepsilon > 0$ was arbitrary, we infer $\varphi(\alpha f^*) = 0$. Using that $\varphi$ is strictly increasing with $\varphi(0) = 0$, we obtain $f^* = 0$, a contradiction to our assumption $f^* > 0$. Hence, $0 =  f  = -g$.

(v) one one implies (i): Let $W \in \mathcal A \cap \ell^\infty(X)$ and $\varphi \in \mathcal I$ for which (v) is satisfied. The stability result Proposition~\ref{proposition:stability of stochastic completeness at infinity} yields that (v) is satisfied for any $\varphi' \in \mathcal I$ with $\varphi'(t) \geq \varphi(t)$, $t \geq 0$. Hence, we can assume without loss of generality, that $\lim_{t \to \infty} \varphi(t) = \infty$. 

 Assume $b$ over $(X,m)$ is not stochastically complete. Our unboundedness assumption on $\varphi$ implies that if $f \in \cF$ satisfies $\cL f \geq 0$, then $f \in \cF_\varphi$. Below we construct a nontrivial $0 \leq h \in \ell^\infty(X)$ with $\cL h + \varphi(Wh) \leq 0$. For $g  = -h$ this implies $\cL g \geq \varphi(-W g) \geq 0$, which shows $g \in \cF_\varphi$ and $\varphi\inv(\cL g) + W g \geq 0$, a contradiction to (v). 

We let $\beta = \sup_{x \in X} W(x) > 0$ and use the continuity of $\varphi$ to choose $\alpha > 0$ such that $\varphi(t) < \alpha/2$ for all $0 \leq t \leq \beta$. 

According to Proposition~\ref{proposition:stochastic completeness} (applied with $c = 0$) there exists  $0 \leq f \in \ell^\infty(X)$ with $\av{f}_\infty = 1$ and $\cL f \leq -\alpha f$.  We consider the function $h = (f - 1/2)_+$. Then $h \neq 0$,  $0 \leq h \leq 1/2$ and a direct computation shows 
$$\cL h(x) \leq \begin{cases}
                 \cL f (x) &\text{if } f(x) > 1/2\\
                 0 &\text{if } f(x) \leq 1/2
                \end{cases}.
$$
Moreover, if $f(x) \leq 1/2$, we have $\varphi(W(x) h(x)) = 0$ and if $f(x) > 1/2$, we obtain
$$ \varphi(W(x) h(x)) \leq \varphi(\beta) < \alpha/2 \leq \alpha f(x) \leq -\cL f(x). $$
Put together, this implies $\cL h + \varphi(Wh) \leq 0$, a contradiction to (v). 
 \end{proof}

\begin{remark}
\begin{enumerate}[(a)]
 \item Our characterization shows in particular that if $b$ is stochastically complete, then for all $W \in \mathcal A$ and all $\varphi \in \mathcal I$ the graph $b$ with $W$ and $\varphi$ is stochastically complete at infinity. In order to deduce stochastic completeness from stochastic completeness at infinity for one particular $W \in \mathcal A$ and $\varphi \in I$, we have to additionally assume that $W$ is bounded. Indeed, under mild assumptions on $\varphi$ Corollary~\ref{coro:large potential} shows that even if $b$ is not stochastically complete, one can find a large $W$ making $b$ with $W$ and $\varphi$ stochastically complete at infinity.  
 \item In the case when $W = \alpha$ and $\varphi = {\rm id}_\IR$, assertion (ii) is also known as the $L^\infty$-positivity preserving property, a notion that was introduced in \cite{Gun16}.  That it is related to stochastic completeness was observed in \cite{BM23} for Riemannian manifolds. In the linear situation on  Riemannian manifolds the most delicate part is the question about regularity of bounded functions  satisfying $(-\Delta + \alpha) u \geq 0$ in the sense of distributions, a problem which disappears in the discrete case.
\end{enumerate}

\end{remark}
 
 The one/one statements yield particularly strong non-uniqueness results in the stochastically incomplete case. We only mention one. 
 
\begin{corollary}
 The following assertions are equivalent. 
 \begin{enumerate}[(i)]
  \item The graph $b$ over $(X,m)$ is not stochastically complete. 
  \item For all $W \in \mathcal A \cap \ell^\infty(X)$ and all $\varphi \in \mathcal I$ there exists a nontrivial $0 \leq f \in \ell^\infty(X)$ such that $-\cL f = \varphi(Wf)$. 
 \end{enumerate}
 In this case, for $\alpha > 0$, $W \in \mathcal A \cap \ell^\infty(X)$ and $\varphi \in \mathcal I$ the function $w_\alpha = \alpha - R(\alpha W)$ is the maximal function $0 \leq w_\alpha \leq \alpha$ satisfying $-\cL w_\alpha = \varphi(W w_\alpha)$.  
\end{corollary}
 \begin{remark}
  The previous corollary is the precise analogue of the elliptic characterization of stochastic completeness in \cite[Theorem~1.1]{gri23} (if $W = 1$). Note however, that our considerations go beyond, as we also have a characterization via supersolutions and a theory of maximal solutions. 
 \end{remark}

\bibliographystyle{alpha}
\bibliography{stoch-comp.bib}

\end{document}